\documentclass[11pt]{amsart}
\usepackage{amsmath,amssymb,amscd,fullpage,pb-diagram,hyperref,MnSymbol}

\newcommand{\sbc}{\on{sbc}}
\newcommand{\Flat}{\on{Flat}}
\newcommand{\Asai}{\on{Asai}}

\newcommand{\sInd}{\on{SI}}
\theoremstyle{definition}
\newcommand{\X}{\mathbf{X}}
\newcommand{\un}{{\on{un}}}
\newcommand{\unr}[1]{\X_{#1,\un}}

\newcommand{\s}{{\on{sp}}}
\newcommand{\II}{{I\!\!I}}
\newcommand{\ff}[1]{f^\circ_{[#1]}}
\newcommand{\ord}{\on{ord}}
\newcommand{\ssm}{\smallsetminus}
\newcommand{\iq}[1]{\int\limits_{\quo{#1}}}

\newtheorem{thm}[equation]{Theorem}
\newtheorem{cor}[equation]{Corollary}

\newtheorem{prop}[equation]{Proposition}
\newtheorem{lem}[equation]{Lemma}
\newtheorem{rmk}[equation]{Remark}

\numberwithin{equation}{section}
\newcommand{\wt}{\widetilde}

\newcommand{\bpm}{\begin{pmatrix}}
\newcommand{\epm}{\end{pmatrix}}
\newcommand{\bsm}{\begin{smallmatrix}}
\newcommand{\esm}{\end{smallmatrix}}
\newcommand{\bspm}{\left(\begin{smallmatrix}}
\newcommand{\espm}{\end{smallmatrix}\right)}
\newcommand{\bm}{\begin{matrix}}
\renewcommand{\em}{\end{matrix}}
\newcommand{\bbm}{\begin{bmatrix}}
\newcommand{\ebm}{\end{bmatrix}}

\newcommand{\bs}{\backslash}

\newcommand{\C}{\mathbb{C}}
\newcommand{\G}{\mathbb{G}}
\newcommand{\A}{\mathbb{A}}

\newcommand{\Z}{\mathbb{Z}}

\newcommand{\Ad}{\operatorname{Ad}}

\newcommand{\Gal}{\operatorname{Gal}}
\newcommand{\diag}{\operatorname{diag}}

\newcommand{\der}{\operatorname{der}}
\newcommand{\Ind}{\operatorname{Ind}}

\newcommand{\Fr}{\operatorname{Fr}}

\newcommand{\la}{\langle}
\newcommand{\ra}{\rangle}

\newcommand{\ve}{\varepsilon}
\newcommand{\vph}{\varphi}

\newcommand{\on}{\operatorname}
\newcommand{\ol}{\overline}

\renewcommand{\Re}{\on{Re}}
\newcommand{\Res}{\on{Res}}

\newcommand{\gm}{\gamma}

\newcommand{\quo}[1]{#1(F)\bs #1(\A)}

\renewcommand{\c}{\mathcal}
\newcommand{\f}{\mathfrak}

\begin{document}
\author{Joseph Hundley} \email{jhundley@math.siu.edu}
\address{Math. Department, Mailcode 4408,
Southern Illinois University Carbondale, 1245 Lincoln
Drive Carbondale, IL 62901}

\title{Holomorphy of adjoint $L$ functions for quasisplit $A_2$}
\date{\today}
\maketitle
\tableofcontents

\section{Introduction}
Let $F$ be a number field\footnote{It seems that most 
of these results should also hold in the function field 
case.}. Let $G$ be a quasisplit $F$-group, isomorphic 
to $SL_3$ over $\ol F.$ Thus, $F$ is either $SL_3,$
or a quasisplit unitary group attached to some 
quadratic extension  $E/F.$ It's $L$-group is then 
$PGL_3(\C)$ in the split case, 
or the semidirect product of $PGL_3(\C)$ with $\Gal(E/F)$
in the nonsplit case.
In either case we have an action of 
$^LG$ on $PGL_3(\C)$ by conjugation, which may be regarded as an action on $GL_3(\C)$ which fixes the 
center. This then induces an action on the Lie 
algebra $\f{sl}_3(\C)$ which we denote $\Ad.$ 
Note that in the nonsplit 
case this does not coincide with the 
definition of 
$\Ad$ in \cite{su21adjoint}. Rather, the 
nontrivial element $\Fr$
of $\Gal(E/F)$ will act by $X \mapsto -_tX,$
as this is the differential of its action on $PGL_3(\C).$
Let $\Ad'$ denote the representation 
of $^LG$ considered in \cite{su21adjoint}, where
$\Fr$ acts by $X \mapsto \,{}_tX.$

Let $\pi$ be a globally generic irreducible cuspidal 
automorphic 
representation of $G.$ 
Then we consider the adjoint $L$ function 
$L(s, \pi, \Ad).$ 
One would like understand the poles of this $L$ function. 
We discuss an attack based on 
the integral representation given 
in \cite{SL3Adjoint}, \cite{su21adjoint}
and a strengthening the 
results of \cite{Ginzburg-Jiang}. 
Our proof also applies to certain twisted $L$ 
functions.

We briefly review the local zeta integral 
for $L(s, \pi, \Ad)$ presented in
 \cite{SL3Adjoint}, \cite{su21adjoint}.
 First, one fixes 
  an embedding of $G$ 
into the split exceptional group of type $G_2.$
Let $P$ be the maximal standard parabolic subgroup of $G_2,$ 
whose Levi contains the root subgroup attached to the short root.
Let 
$f$ be a flat section of the family of induced representations 
attached to a family of characters of $P(F)\bs P(\A),$ and 
$E_P$ be the corresponding Eisenstein series operator 
(these notions are reviewed in sections \ref{sec: Flat Sections} and \ref{sec:Eis ser}).

One may then define
\begin{equation}
\label{eq: def of I(s, vph, f)}
I(s, \varphi, f) := 
\iq{G} \vph(g) E_P.f(g,s) \, dg, 
\end{equation}
 where $s \in \C,$ and 
$\varphi$ is a generic cusp form $\quo{G} \to \C.$ 
Both  \cite{SL3Adjoint}, and \cite{su21adjoint}
present the argument under the assumption that the characters
are unramified, but the extension to the general case is direct 
so we may regard the theorem as proved in the 
ramified case as well. 
The integral $I(s, \varphi, f)$ unfolds to a new integral 
$I(s, W_{\varphi}, f)$ where $W_{\varphi}$ is the Whittaker
function attached to $\varphi.$ 
Assuming that $W_{\varphi}$ and $f$ are factorizable, 
$I(s, W_{\varphi}, f)$ then factors as a product 
of local zeta integrals
$I(s, W_v, f_v).$ 
Here $W_v$ and $f_v$ are the local components at a place
$v$
of $W_{\varphi}$ and $f$ respectively.
Moreover, 
if $W_v,f_v$ and $\chi_v$ are all unramified,
and $W_v$ and $f_v$ are normalized,  
then 
$$I(s, W_v, f_v) = \frac{L(3s-1, \pi_v, \Ad'\otimes \chi_v)}{L(3s, \chi_v)L(6s-2, \chi_v^2) L(9s-3, \chi_v)}.\qquad (\text{See note}\footnote{
This corrects a typo which appears in \cite{su21adjoint}. In that paper, $\zeta(3s-9)$ appears where $\zeta(9s-3)$ 
would be correct.})
$$ 
Hence, 
\begin{equation}
\label{global zeta int = L(Ad)/Norm fact x finite prod}
I(s, \varphi, f)
= \frac{L^S(3s-1, \pi, \Ad'\otimes \chi)}{L^S(3s, \chi)L^S(6s-2, \chi^2) L^S(9s-3, \chi)} \prod_{v \in S} I(s, W_v, f_v),
\end{equation}
where $S$ is a finite set of places, away from which 
$W_v, f_v$ and $\chi_v$ are unramified.
Notice that $L(s, \pi, \Ad) = L(s, \pi, \Ad')$ 
if $G$ is split and 
$L(s, \pi, \Ad' \otimes \chi_{E/F}),$
where $\chi_{E/F}$ is the quadratic character 
attached to the extension $E/F$ by class field 
theory if $G$ is not split. 

One expects that in general the $L$ function $L(s, \pi, \Ad \otimes \chi)$ 
should be entire. 
Indeed, in the split case  
$L(s, \pi, \Ad\times \chi ) = L(s, \pi \otimes \chi \times \wt \pi)/L(s,\chi),$ and it follows that the possible 
poles are precisely the zeros of $L(s, \chi),$ unless $\chi$ is nontrivial and  $\pi \otimes \chi \cong \wt \pi,$
in which case there are additional simple poles at $s=0$ and $s=1.$ One expects that
$L(s, \pi \otimes \chi \times \wt \pi)$ is divisible by $L(s, \chi)$ and hence that 
the only actual poles are the  simple poles at $s=0$ and $s=1$ which occur when 
 $\chi$ is nontrivial and  $\pi \otimes \chi \cong \wt \pi.$ In the special 
 case when $\pi = \otimes'_v \pi_v$ and at least one component $\pi_v$
 is supercuspidal, this was proved by Flicker
 \cite{Flicker}.

In the nonsplit case, one must replace $L(s, \pi \times \wt \pi)$ with the Asai $L$ function of the stable 
base change lifting of $\pi$. One must also account for the image of certain theta liftings. 
Indeed, consider 
\begin{equation}
\label{GL2 x GL2 --> GL3}
\left \{ 
\bpm a& &b \\ &t& \\ c&&d \epm 
\right \} \cong GL_2(\C) \times GL_1(\C) \subset GL_3(\C)
\end{equation}
This subgroup is stable under the outer automorphism 
which realizes the action of the nontrivial Galois element of the $L$ group. Thus we obtain a subgroup
of the $L$ group.
This subgroup may be realized the image of an $L$-homomorphism from 
 the $L$ group of a product of smaller unitary groups $U_{1,1} \times U_1$
(still attached to the same quadratic extension). 
See \cite{RogawskiBook}.
The above subgroup clearly stabilizes the one dimensional 
subspace of $\f{sl}_3(\C)$ spanned by $\diag(1, -2, 1),$ and so does the map $X \mapsto \ _t\!X:=J\ ^t\!X J.$ \
Thus we obtain a one dimensional space stable under the restriction of the representation $\Ad'$ to $^L(U_{1,1} \times U_1),$ 
so that $L(s, \pi, \Ad')$ should have a pole whenever $\pi$ is a lift from $U_{1,1} \times U_1.$ 
(In the split case, this issue does not arise because an element of the image of the corresponding lifting can never be cuspidal). 

An approach to controlling poles of these $L$ functions, 
which is based on the study  of $I(s, \varphi ,f)$ and does 
not depend on any property of a local component of $\pi$ was 
pioneered in \cite{Ginzburg-Jiang}.
 If $I(s, \varphi, f)$ has a pole at $s=s_0$
then each of its negative Laurent coefficients is a global integral,  
similar to $I(s, \varphi, f),$ but with the 
Eisenstein series replaced by its corresponding negative Laurent coefficient. 
Thus, it suffices to show that the negative Laurent coefficients of the 
Eisenstein series 
used in $I(s, \varphi, f),$
when restricted to the subgroup $G \hookrightarrow 
G_2$ 
are ``orthogonal''\footnote{``Orthogonal'' is in quotes because the residues need 
not be $L^2.$ But one can extend the inner product
on $L^2(G(F) \bs G(\A))$ to allow pairing cusp forms
with arbitrary smooth functions of moderate growth.} to cusp forms. Following \cite{Ginzburg-Jiang}, this can be done 
by expressing such Laurent coefficients in terms of 
Eisenstein series 
induced from characters of the {\it other} maximal 
parabolic subgroup of $G_2,$ 
and then checking that the restrictions of 
these Eisenstein series 
are ``orthogonal'' to $G$-cuspforms.
In 
\cite{Ginzburg-Jiang}, 
it is shown that the Eisenstein series appearing in the construction 
of $I(s, \varphi, f)$ for the case of trivial $
\chi$ has only two poles in $\Re(s) > \frac 12, $
with one being simple and the other double. The residue of the simple pole is 
a constant function and thus obviously orthogonal to cusp forms. 
At the double pole, a ``first term identity'' is proved, which 
expresses the leading term of the Laurent expansion in 
terms of an Eisenstein series from the other parabolic. This 
rules out a double pole of the adjoint $L$ function. In order
to rule out a simple pole by this method, one would need a ``second 
term identity.'' In 
section \ref{sec: identity of unram Eis ser} we prove an identity of this 
type. It is my understanding that such an identity was first obtained by 
Jiang in unpublished work. .

That being said, if the second pole of the Eisenstein series
gave rise to a pole of the global adjoint $L$ function 
$L(s, \pi,\Ad),$ this pole would  
occur at $s=1.$   
Such a pole is impossible, because 
$L(s, \pi, \Ad) = L(s, \pi \times \wt \pi) /\zeta(s),$ 
and  $L(s, \pi \times \wt \pi)$ and $\zeta(s)$ both have simple poles 
at $s =1.$

In this paper we pursue the approach pioneered by Ginzburg and Jiang. 
First, we analyze the poles of the Eisenstein series in the case of 
nontrivial $\chi.$ This allows us to deduce information about 
the poles of the local zeta integral.
We also prove a key vanishing result needed to 
deduce holomorphy of $L(s, \pi, \Ad \times \chi)$ 
at $\Re(s) = \frac12$ from \eqref{global zeta int = L(Ad)/Norm fact x finite prod}.
Then, we prove a weak result regarding local zeta integrals
at ramified and Archimedean primes. While preparing this
manuscript, I have learned that a stronger result-- a local 
functional equation-- has been obtained by Qing Zhang. 
The weaker result proved here suffices for our application, 
permitting us to deduce that each pole of the partial adjoint $L$ 
function must be a pole of the global zeta integral for some choice
of data. 

Our main result is theorem \ref{thm: main}, which states 
that, in the split case, every pole of $L(s, \pi, \Ad \times \chi)$ in the half plane 
$\Re(s) \ge \frac 12$ is simultaneously a zero of $L(s, \chi)$ 
and a pole of the finite product over the ramified and 
Archimedean places. 
Using this result, together with knowledge of the 
form of the Gamma factor and the zeros of the 
Riemann zeta function, Buttcane and Zhou were able to show 
holomorphy of the complete adjoint $L$ function  (and hence also all partial $L$ functions)
for an $SL(3, \Z)$ Maass form with trivial central character 
(such a form generates a representation unramified at 
all finite places).

\subsection{Acknowledgements}
Thanks to Richard Taylor for pointing out the problem with the definition of 
``Ad'' in \cite{su21adjoint}, 
to Sol Friedberg for pointing out the reference \cite{Flicker}, 
to Dihua Jiang for helpful explanation of \cite{Ginzburg-Jiang}, 
and for letting me know about his unpublished work,
to Paul Garrett 
and David Loeffler for helpful explanations on 
MathOverflow, and to Jack Buttcane, Fan Zhou, and Dorian Goldfeld
for stimulating discussions.

\section{Induced representations, interwining operators, and their poles}
\label{s: ind reps, int ops, poles}

\subsection{Characters and degenerate induced representations}
Let $F$ be a number field as before. 
If $G$ is an $F$-group,  write 
$X(G)$ for the group of rational characters
of $G$ and 
$\X_G$ for the complex manifold
 of characters of $G(\A)$ trivial on $G(F).$ 
These are groups and we write them 
additively. To reconcile with multiplicative notation for $G(\A)$ and $\C^\times$, we use an exponential
notation for the characters: the value of $\chi \in \X_G$ at $g \in G(\A)$
is denoted $g^\chi.$  A similar notation is used for cocharacters.
We identify $\chi \in X(G)$ with the character 
of $G(F) \bs G(\A)$ obtained by composing it 
with the absolute value on $\A^\times.$ 
This extends to a mapping of $X(G) \otimes_\Z \C$ 
into $\X_G.$ The image is the set of unramified 
characters, which we denote $\unr{G}.$
Similarly we denote the complex manifold
of all characters of $G(F_v)$ by 
$\X_{G,v}$ and the image of 
$X(G) \otimes_\Z\C$ in it by $\X_{G,v,\un}.$ 
We identify $\X_{GL_1,\un}$ with 
$\C$ using the map $s \to |\ |^{s}.$
If $\varphi^\vee$ is a cocharacter, then 
$\la \varphi^\vee, \chi \ra \in \X_{GL_1}.$

We shall only require split connected reductive $F$-groups
 with simply connected derived groups.
We always assume that each is equipped with a choice of split torus and Borel containing it.
The torus is denoted $T$ and the Borel $B.$
Let $G$ be such a group and $M$ a standard Levi. 
Then we may identify $\X_M$ with $\{ \chi \in \X_T : \la \chi, \alpha^\vee \ra  = 0, \alpha \in \Phi(T,M)\}.$
Likewise, we may identify $\X_{M,\un}$ 
with $\{ \chi \in \X_{T,\un} : \la \chi, \alpha^\vee \ra  = 0, \alpha \in \Phi(T,M)\}.$

We would like to choose a complement 
$\X_{G,0}$ 
to $\X_{G,\un}$ in $\X_G.$ 
When $G=GL_1$ this is done by 
taking the normalized characters, 
i.e., those that are trivial on the multiplicative group 
of positive reals, embedded diagonally at all the infinite 
places.
When $G$ is a torus, it can be identified with several
copies of $GL_1$ by choosing a $\Z$-basis
for $X(G)$ and the subgroup $\X_{G,0}$ 
thus obtained is independent of the choice. 
If $G  = G_{\der} T$ 
where $G_{\der}$ is the derived group and 
$T$ is a torus, 
then restriction gives 
an embedding $\X_G \hookrightarrow \X_T,$
and we may apply the decomposition 
$\X_T = \X_{T,\un} \oplus \X_{T,0}$ to 
obtain the corresponding decomposition of 
$\X_G.$

For archimedean local fields, we again define a character
to be normalized if it is trivial on the positive
reals.
For nonarchimedean fields, we first choose 
a uniformizer and then say that a character
is normalized if it is trivial on the uniformizer.
This leads to 
 similar decompostions
$\X_{G,v} = \X_{G,v,\un} \oplus \X_{G,v,0}$ 
 into unramified characters and normalized
 characters. For any character $\chi,$
 we define $\chi_{\un}$ and $\chi_0$ 
 to be the components relative to this decomposition.
If $\chi = s + \chi_0 \in \X_{GL_1}$ (resp. $\X_{GL_1, v}$)
we define $L(\chi)$ to be the usual global 
(resp. local) $L$ function $L(s, \chi_0)$ 
(keeping in mind that $\X_{GL_1, \un}$ has 
been identified with $\C$).

For $H$ a $T$-stable $F$-subgroup we write $\Phi(T,H)$ for the roots of $T$ in $H.$
The Weyl group is denoted $W.$ It is realized as a quotient of the normalizer, $N_G(T),$ of $T$ in $G.$ 
We also assume $G$ equipped with a realization, i.e.
a  family of isomorphisms
$\{ x_{\alpha}: \G_a \to U_\alpha \} _{\alpha \in \Phi(G,T)},$
such that  
$x_{\alpha}(1)x_{-\alpha}(-1) x_{\alpha}(1) \in N_G(T)$
for each $\alpha.$  This product is then a representative
for the simple reflection attached to $\alpha$ and 
one may select representatives for other elements
of the Weyl group using them.

Take $P$ a parabolic with Levi $M$ and $\chi \in \X_M.$
Write $\rho_P = \frac 12 \sum_{\alpha \in \Phi(P,T)} \alpha.$ 
We define $I_P^G(\chi)$ to be the normalized $K$-finite induced 
representation of $G(\A)$ and $\Ind_P^G(\chi)$ to be the non-normalized version.
For $\chi \in \X_{M,v}$ we define 
$I_P^G(\chi)$ to be the normalized $K_v$-finite
induced representation of $G(F_v)$ and $\Ind_P^G(\chi)$ the non-normalized version. 
In either case,  $I_P^G(\chi) = \Ind_P^G( \chi + \rho_P),$ and  
 is a subset of $I_B^G(\chi + \rho_P - \rho_B).$
In the important special case when the Levi of $P$ is rank one with unique root $\alpha,$
this becomes $I_B^G( \chi - \frac \alpha2).$ 

\subsection{Flat Sections}\label{sec: Flat Sections}
Fix a reductive group $G$ and standard Levi $M,$
and a normalized character $\chi_0 \in \X_M.$
We consider the family of induced representations
$I_P^G(\chi)$ with $\chi$ in $\chi_0+ \unr M.$
Denote the family as a whole by $\c I_P^G(\chi_0).$ 
By a section
 we mean a function 
$$\chi \mapsto f_\chi, \qquad (\chi \in \chi_0+ \unr M)$$
such that (1) $f_\chi \in I_P^G(\chi)$ for each $\chi\in \chi_0+  \unr M,$ 
(2) $f$ is smooth as a function $X_M \times G \to \C.$
We say that $f$ is flat if $f( \chi_0+s, k)$ is independent
of $s \in \X_{M, \un}$ for $k \in K.$ 
(Here $K$ is a fixed maximal compact subgroup.)  
The set of flat sections of $\c I_P^G(\chi_0)$ is a complex vector 
space. Denote it $\Flat(\chi_0).$

\subsection{Coordinates on $\X_T$ in the case of $G_2$}
Our main results deal with induced representations on the split 
exceptional group $G_2.$ 
I write $\alpha$ for the short root and $\beta$ for the long root. 
{\it Unfortunately, this is the opposite of the notation used in \cite{Ginzburg-Jiang}.}
I write $U_\gm$ for the root subgroup attached to any root $\gm.$
I assume $G_2$ to be equipped with a choice of Borel and of maximal torus. 
These are $B$ and $T.$ 
I write $P=MU$ for the standard parabolic subgroup whose Levi contains $U_\alpha$
and $Q=LV$ for the one whose Levi contains $U_\beta.$
For $\chi_1, \chi_2 \in \X_{GL_1}$ let 
$[\chi_1, \chi_2]$ denote the element of $\X_T$ which satisfies
$$
\la [\chi_1, \chi_2],\  \alpha^\vee \ra = \chi_1,\qquad
\la [\chi_1, \chi_2],\  \beta^\vee \ra = \chi_2.
$$
Thus $\varpi_1:=[1,0]$ and $\varpi_2:=[0,1]$ are the two fundamental weights.
Note that $[\chi_1, \chi_2] \in \X_M \iff \chi_1 = 0$ 
and $[\chi_1, \chi_2] \in \X_L \iff \chi_2 = 0.$

\subsection{Normalization and poles of intertwining operators: $GL_2$ case}
\label{ss: normalization and poles of int ops GL2 case}
We study poles of intertwining operators. The theory is fairly uniform for split groups, 
and reduces to the special case of $GL_2.$ 
First we consider the case of $GL_2.$ 
Write $B_{GL_2}$ for the standard Borel of $GL_2$ consisting of upper triangular matrices.
Take $\chi = \prod_v \chi_v$ a character of $B_{GL_2}(\A).$
 Let $w$ be the unique nontrivial element 
of the Weyl group, and $\alpha$ the unique positive root.
\begin{lem}
The normalized local intertwining operator
$$
M^\star_v(w, \chi_v):=
\frac{1}{L_v(\la \chi_v , \alpha^\vee \ra)}
M_v(w, \chi)$$
extends analytically to all of $\X_T.$
When $\la \chi_v \circ \alpha^\vee \ra$ is unramified 
$I_{B_{GL_2}}^{GL_2}(\chi_v)$ 
has a normalized ``spherical''\footnote{By ``spherical'', we mean fixed by 
the intersection of $SL_2(F_v)$ 
with the maximal compact subgroup.}
 vector which we denote $f^\circ_{\chi_v}.$
 Then 
 $$M_v(w, \chi_v).f^\circ_{\chi_v} = \frac{L(\la \chi_v, \alpha^\vee \ra)}{L(\la \chi_v, \alpha^\vee \ra + 1 )}f^\circ_{w \chi_v}.$$ 
\end{lem}
\begin{proof}
In the nonarchimedean case, both assertions can be verified by fairly direct computations. Alternatively, the 
first assertion is a special case of a result of Winarsky, \cite{Winarsky}, and the second assertion is a special 
case of the result in section 4 of Langlands, {\it Euler products} \cite{eulerproducts}. 
Over the reals, both assertions can be deduced from proposition 2.6.3 of \cite{Bump-GreyBook}.
The second assertion is also the simplest case of the formula of Gindikin and Karpalevic \cite{Gindikin-Karpalevic}, 
first proved for $GL_n$ by Bhanu Murti \cite{BhanuMurti}. 
Over the complex numbers, both assertions follow from Lemma 7.23 of \cite{Wallach-corvallis}. See also 
\cite{Garrett} and additional references therein.
The first assertion over either archimedean field also follows from the 
generalization found on p. 110 of \cite{Shahidi}.
\end{proof}
Consequently if $f = \prod_{v \notin S} f_{\chi_v}^\circ \prod_{v\in S} f_v $
is an element of the global induced space $I_{B_{GL_2}}^{GL_2}(\chi)$, then 
$$
M(w, \chi).f = \frac{L(\la \chi, \alpha^\vee \ra) }{L(\la \chi, \alpha^\vee \ra+1) }\prod_{v \notin S} f_{\chi_v}^\circ \prod_{v\in S}  L_v(\la \chi_v, \alpha^\vee \ra+1) M_v^\star(w, \chi_v) f_v.
$$
We deduce that poles of $M(w, \chi)$ come in three classes
\begin{enumerate}
\item Poles of $\prod_{\alpha> 0, \ w\alpha < 0} L(\la \chi, \alpha^\vee \ra),$
which are not cancelled by poles of $\prod_{\alpha> 0, \ w\alpha < 0} L(\la \chi, \alpha^\vee \ra+1),$
These occur at $\la \chi, \alpha^\vee \ra =1,$ and in particular do not 
occur unless $\la \chi, \alpha^\vee \ra _0=0.$
\item Zeros of $L(\la \chi, \alpha^\vee \ra+1).$ These are all in the strip $-1< \Re \la \chi, \alpha^\vee \ra _{\un} <0.$
\item Poles of $\prod_{\alpha> 0, \ w\alpha < 0} \prod_{v\in S}L_v(\la \chi_v, \alpha^\vee \ra+1).$
These are all in the half plane $\Re \la \chi, \alpha^\vee \ra_{\un} \le -1.$ 
\end{enumerate}

\subsection{Poles of intertwining operators
on the principal series: general case}
Now let $G$ be a general split reductive group, $B$ 
its Borel, $\chi= \prod_v \chi_v$ a character of $B(\A),$ and $w$ 
any element of the Weyl group. For each root $\alpha$
we have a map $SL_2  \to G$ and can decompose the 
standard intertwining operator 
$M(w, \chi)$ as a composite of intertwining operators indexed by 
$\{ \alpha > 0 : w\alpha < 0 \}.$   Poles of the intertwining operator 
attached to a root $\alpha$ are of the same three types, along 
hyperplanes $\la \alpha^\vee, \chi \ra = c$ in the space $\X_{B, \un}$ 
defined using the corresponding coroot.

\subsection{Eisenstein series}\label{sec:Eis ser}
Now $G$ is a reductive group and $P$ a parabolic subgroup.  
We fix a suitable maximal compact subgroup 
$K = \prod_v K_v$ of $G(\A)$ and 
let $\c A(G)$ denote the space 
of automorphic forms (relative to $K$)
$G(\A) \to \C,$
that is, the space of smooth functions $\phi: \quo G \to \C$ 
of moderate growth which are finite under the action of 
$K$ and the center, $\f z_G$ of the universal enveloping algebra 
of the Lie algebra of $G(F_\infty).$ 
 
Fix $\chi_0 \in \X_{M,0}$ and 
let $\Flat(\chi_0)$ denote the space of flat 
sections of $\c I_P^G(\chi_0).$ 
For $f \in \Flat(\chi_0),$ we define the Eisenstein series 
$$
E_P.f: G(\A) \times (\chi_0 + \X_{M,\un}) \to \C
$$
by 
$$
E_P.f(g,\chi) = \sum_{\gm \in P(F) \bs G(F) } f_\chi(\gm g)
$$
for values of $\chi$ such that this sum is convergent and by 
meromorphic continuation elsewhere. 
Outside the domain of convergence, one encounters poles of finite order along 
a locally finite set of root hyperplanes. For each $\chi$ away from the poles, 
$f \mapsto E_P.f(\cdot, \chi)$
is an intertwining operator
$I_P^G(\chi) \to \c A(G).$  We denote it 
$E_P(\chi).$

\subsection{Normalization of $G_2$ Eisenstein series} 
We briefly recall the Eisenstein series which appears in \cite{SL3Adjoint} and \cite{su21adjoint} and its 
normalization. The Eisenstein series in question are attached to the parabolic $P=MU$ as in section \ref{sec: identity of ramified Eisenstein series}.
 In \cite{SL3Adjoint} and \cite{su21adjoint}, unramified Eisenstein series are considered. The space $\unr M$ is one dimensional and can 
 conveniently be identified with $\C$ using the mapping $s \mapsto \delta_P^s.$  
 Here $\delta_P$ is the modular quasicharacter. 
 In the notation of 
 section \ref{sec: identity of ramified Eisenstein series}, $\delta_P^s = [0, 3s].$ Equivalently, the half-sum of the roots of $P$ is   $\rho_P = [0,\frac 32].$ Thus 
 $\Ind_P^G(\delta_P^s) = I_P^G([0, 3s-\frac 32]) \subset I_B^G[-1, 3s-1]).$ In order to generalize the construction of \cite{SL3Adjoint} and \cite{su21adjoint} 
 to get $L(s, \pi, \Ad' \times \chi)$ for general $\chi,$ we would use 
 $I_P^G([0, 3s- \frac 32 + \chi]),$

\subsection{Application to relevant intertwining operators for $G_2$}
We apply this to the intertwining operators that appear in the 
constant term of our $G_2$ Eisenstein series. 
Write $c(u, \chi_0) = \frac{L(u, \chi_0)}{L(u+1, \chi_0)}.$ 
For $\chi \in \X_T$ and $w \in W$ let $c(w, \chi)= \prod_{\alpha> 0, \ w\alpha < 0 } c(\la \alpha^\vee, \chi\ra).$ 
Notice that this expression appears in our formula for $M(w, \chi)$ above. 
The constant term of our Eisenstein series may be expressed as a sum over $w \in W$ such that $w\alpha > 0.$
Here $\alpha$ is the short simple root. 
There are six such $w,$
one of each length from $0$ to $5.$ 
We write $w_i$ for  the element of length $i.$ 
If $\chi = [-1,3s-1+\chi_0]$ and
$$
\c C:= (c( 3s-1, \chi_0) ,c(9s-4, \chi_0^3) ,c(6s-3, \chi_0^2) ,c(9s-5, \chi_0^3) ,c(3s-2, \chi_0)),
$$
then $c(w_i, \chi) = \prod_{j=1}^i \c C_j.$
\begin{enumerate}
\item 
Poles which arise from the pole of the zeta function at $1$ will occur 
at $2/3$ and $1$ if $\chi_0$ is trivial; at $5/9$ and $2/3$ if $\chi_0^3$ is trivial, 
and at $2/3$ if $\chi_0^2$ is trivial. This is to be interpreted additively, 
so if $\chi_0$ is trivial then the pole at $2/3$ can be a triple pole. Otherwise it 
is simple.
\item Poles which arise from zeros of a global $L$ function are in the 
half plane $\Re(s) < \frac 12.$ For example, $c(9s-5, \chi_0^3)$
could have poles as far right at $\Re(s) = 5/9-\ve,$ but it
never occurs without $c(9s-4, \chi_0^3)$ so the zeros of the $L$ function 
in its denominator are always cancelled. We only get poles from the zeros of $L(9s-3, \chi_0^3),$
and these begin at $4/9.$ 
\item 
Poles which arise from local $L$ functions are all in the half 
plane $\Re(s) \le 4/9.$
\end{enumerate}
It follows that the only possible poles of our Eisenstein series in the half plane $\Re(s) \ge 1/2$ are 
at $5/9, 2/3$ and $1,$ and can occur only if $\chi_0$ is trivial, quadratic, or cubic. 
Notice that the pole at $2/3$ would correspond to a pole of $L(s, \pi, \Ad' \times \chi)$ 
at $1.$ 
More detailed information regarding the poles of the intertwining operators
in $\Re(s) \ge \frac 12$ is recorded in the following table.
$$
\begin{array}{|c|c|c|c|}
\hline 
\chi & \text{ trivial } & \text{ nontrivial quadratic } & \text{ nontrivial cubic }\\
\hline 
w_0 & \text{ holomorphic } &  \text{ holomorphic } &  \text{ holomorphic }\\
w_1 & 2/3&  \text{ holomorphic }&  \text{ holomorphic }\\
w_2 & 2/3,  \ 5/9 &  \text{ holomorphic } & 5/9\\
w_3 & 2/3 (\text{double}), \ 5/9  & 2/3 & 5/9\\
w_4 & 2/3 (\text{triple}),\ 5/9 & 2/3 & 2/3, 5/9 \\
w_5 & 2/3 (\text{triple}),\ 5/9, \ 1 & 2/3 & 2/3, 5/9\\ \hline
\end{array}
$$
\subsection{Application to the constant
term of our $G_2$ Eisenstein series}
\label{ss: constant terms}
Now, the poles of the Eisenstein series (and their orders) are the same as 
the poles of the constant term (and their orders), which is a 
sum of intertwining operators. 
Having determined the poles of the summands, and their orders, the 
next step is to account for the possibility of cancellation in the sum.
In the unramified case this is done in \cite{Ginzburg-Jiang}.
\begin{thm}[Ginzburg-Jiang]
\label{thm: poles of EP, chi0 trivial}
Assume that $\chi_0$ is trivial. Then 
$E_P$ has a simple pole at $s=1,$ and 
a double pole at $s=2/3.$
At $s=5/9$ it is holomorphic.
\end{thm}
We sketch a proof which is slightly different than the one given in \cite{Ginzburg-Jiang}. 
The technique is similar to \cite{Hundley-Miller} and will be worked out in detail for the ramified case below. We consider the two terms which 
have triple poles at $s=2/3.$ They correspond to the Weyl elements $w_4$ and $w_5.$
We may write $M(w_5, \chi) = M(s_\beta, w_4\chi) M(w_4, \chi).$ 
Here, $s_\beta$ is the simple reflection in the Weyl group attached to the long root $\beta$ ($s_\alpha$ is defined similarly).
We check that  when $\chi_0$ is trivial, $\la \beta^\vee, w_4 \chi\ra $ vanishes at the point corresponding to $s=2/3$ 
It follows from \cite[proposition~6.3]{Keys-Shahidi}
that $M(s_\beta, w_4 \chi)$ is the scalar operator $-1$ at this point. 
Hence $M(w_4, \chi) + M(w_5, \chi)$ is equal to the composition of $M(w_4, \chi)$ 
and an operator which vanishes at $s=2/3.$ 
Similarly, the four terms which give poles at $5/9$ form two pairs such that the sum of each pair is holomorphic 
at $s=5/9.$

\begin{thm}\label{thm: poles in cubic and quadratic case}
Assume that $\chi_0$ is nontrivial quadratic or cubic.
Then $E_P$ has a simple pole at $\frac 23$ and is otherwise holomorphic in $\Re(s) \ge \frac 12.$ 
\end{thm}
\begin{proof}
We have 
$$
M(w_5, \chi) 
= M(s_\beta, w_4\chi)M(s_\alpha, w_3\chi) 
M(s_\beta, w_2\chi) M(s_\alpha, w_1\chi) 
M(s_\beta, \chi)
$$
The corresponding expression for $w_i$ with $i< 5$ is obtained by taking only the rightmost operators in this composite. 

We tabulate key data. First, the spaces that the six operators map into 
$$
\begin{array}{|c|c|c|c|c|}\hline 
M(\dots , \chi) & \text{ Maps to }I_B^{G_2}(\dots)& (s = 5/9 , 3\chi_0  = 0)& (s=2/3, 3 \chi_0 = 0) & (s = 2/3 , 2\chi_0 = 0)\\\hline
w_ 0 & [-1 ,\chi_0 + 3s-1 ] &[-1, \chi_0 + 2/3] & [-1, \chi_0 + 1 ]& [-1, \chi_0 + 1]\\
w_1 & [ 3\chi_0 + 9s-4 ,-\chi_0 - 3s+ 1  & [1, -\chi_0 - 2/3] & [2, -\chi_0 -1] & [\chi_0 + 2, \chi_0 -1]]\\
w_2 & [ -3\chi_0 - 9s+4, 2\chi_0 + 6s-3 ]& [-1, -\chi_0 + 1/3]& [-2, -\chi_0 + 1] & [\chi_0-2, 1]\\
w_3 & [ 3\chi_0 + 9s-5, -2\chi_0 - 6s+3] & [0, \chi_0 -1/3]  & [ 1,\chi_0 -1] & [\chi_0 +1, -1]  \\
w_4 & [ -3\chi_0 - 9s+5, \chi_0 + 3s-2]&[0, \chi_0 - 1/3] &[-1, \chi_0] & [\chi_0-1, \chi_0]\\
w_5 & [ -1,-\chi_0 - 3s+2 ] & [-1, -\chi_0 + 1/3] & [-1 , -\chi_0]& [-1, \chi_0]\\\hline 
\end{array}
$$
And next the elements of $\X_{GL_1}$ which will determine the poles of the rank one operators. 
$$
\begin{array}{|c|c|c|c|c|}\hline 
\text{ pairing }& \text{ general }  & (s = 5/9 , 3\chi_0  = 0)& (s=2/3, 3 \chi_0 = 0) & (s = 2/3 , 2\chi_0 = 0)\\\hline
\la \beta^\vee, \chi \ra   &\chi_0+3s-1&\chi_0+2/3&\chi_0+1&\chi_0+1\\
\la \alpha^\vee, w_1\chi \ra &3\chi_0 + 9s-4&1&2&\chi_0+2\\
\la \beta^\vee, w_2\chi \ra   &2\chi_0 +6s-3&-\chi_0 +1/3&-\chi_0 + 1&1\\
\la \alpha^\vee, w_3\chi \ra &3\chi_0 + 9s-5&0 &1&\chi_0 + 1\\
\la \beta^\vee, w_4\chi \ra   &\chi_0 + 3s-2&\chi_0 - 1/3 & \chi_0& \chi_0\\\hline
\end{array}
$$
The key facts are the following: when the pairing is $1$ the corresponding 
rank-one operator has a pole. When it is zero, the corresponding 
rank-one operator is the scalar operator $-1.$ 
When it is $-1$ the corresponding rank-one operator has a kernel. Otherwise, 
the rank-one operator is an isomorphism.
From this we 
can see that $M(w_i, [-1, \chi_0 + 3s-1])$ has a pole at $s=2/3$ if $2\chi_0 = 0$ and $i \ge 3$ or if $3\chi_0 = 0$ and $i \ge 4.$
In either case, the operators $M(w_i, [-1, \chi_0+1])$ all land in different spaces. Hence there is no possibility of cancellation 
among them and the poles of the individual intertwining operators are inherited by their sum (i.e., the constant term) and then  by the Eisenstein series itself.

We also see that $M(w_i, [-1, \chi_0+3s-1])$ has a simple pole at $s=5/9$ if $i \ge 2$ and $3\chi_0 = 0.$ 
In this case we have two pairs of operators which land in the same space. To study them, set $u=3s-5/3=3(s-\frac59)$ so that $3s-1=u+\frac 23.$ Thus $u$ is a convenient local coordinate in a neighborhood of $s=\frac 59.$  The first key point is that
$$\la \alpha^\vee, w_3[-1, \chi_0+2/3] \ra =0 
\implies M(s_\alpha, w_3 [-1, \chi_0+u+2/3]) = -1 + O(u).$$
Hence 
$$M(w_3, [-1, \chi_0+u+2/3]) + M(w_4, [-1, \chi_0+u+2/3])
= (1+M(s_\alpha, w_3 [-1, \chi_0+u+2/3]) )\circ M(w_3, [-1, \chi_0+u+2/3])
$$
is the composite of an operator with a simple pole at $u=0 $ and  an operator which 
vanishes at $u=0.$ Thus, it is holomorphic at $u=0.$

The second key point is that $M(s_\beta, -s_\beta \chi) = M(s_\beta, \chi)^{-1}.$ (This is an identity of 
meromorphic functions.) 
Hence 
$$\begin{aligned}
H(u):=&M(s_\beta, [-3u, \chi_0+u-1/3])M(s_\alpha, [3u, \chi_0 -2u -1/3])M(s_\beta,  [-3u-1, 2u+1/3-\chi_0] )\\
&=(M(s_\beta, [0, \chi_0 - 1/3])+O(u))(-1+O(u))(M(s_\beta, [-1, 1/3-\chi_0])+O(u))\\
&= -1+O(u).
\end{aligned}$$
But 
$$
M(w_2, [-1, \chi_0 + u + 2/3]) + M(w_5, [-1, \chi_0 + u + 2/3])s=(1+H(u))\circ M(w_2, [-1, \chi_0 + u + 2/3]).
$$
Once again we have an operator with a simple pole composed with an operator that has a zero.
This completes the proof.
\end{proof}
\begin{rmk}
The existence of the pole of $E_P$ at $2/3$ in the cubic case can also be deduced from the 
existence of a pole of the adjoint $L$ function: 
cuspidal representations of $GL_3(\A)$ satisfying $\pi \cong \pi \otimes \chi$ 
exist by Theorem 2.4(iv) of \cite{Clozel-ICM}. Thanks to David Loeffler for explaining this to me.  For 
such a representation $L(s, \pi , \Ad \otimes \chi)= L(s, \pi \times \wt \pi \times \chi)/L(s \chi)$ will have a pole at $s=1.$ As local $L$ functions are nonvanishing
this pole will be inherited by the partial $L$ function, and then, by theorem \ref{thm: nonvanishing of local zetas} below by the global zeta integral.
In the case when $\chi$ is nontrivial quadratic the existence of a pole at $s=2/3$ can also be deduced
 from theorem \ref{thm: identity, quad case} below.
\end{rmk}

\subsection{Key vanishing property of the Eisenstein series}
\begin{prop}\label{prop: vanishing at 1/2 when quad}
Assume that $2 \chi_0=0.$
Then $E_P([-1, 3s-1+\chi_0])$ vanishes 
at $s=\frac 12.$
\end{prop}
\begin{proof}
The proof is similar to that of theorem \ref{thm: poles in cubic and quadratic case}.
Note that 
$w_5[-1, \frac12 + \chi_0] = [-1, \frac 12+ \chi_0],$ $w_4[-1, \frac12 + \chi_0] = w_1[-1, \frac 12+ \chi_0]$
and $w_3[-1, \frac12 + \chi_0] = w_2[-1, \frac 12+ \chi_0].$
Write $w_4 = w_{4,1}w_1$ and $w_3= w_{3,2}w_2.$ 
As in the proof of theorem \ref{thm: poles in cubic and quadratic case}
one readily checks that $M(w_5, [-1, 3s-1 + \chi_0]), 
M(w_{4,1}, w_1 [-1, 3s-1 + \chi_0])$ and $M(w_{3,2}, w_2 [-1, 3s-1 + \chi_0])$
are all $-1+ O(s-\frac12)$ at $s=\frac 12.$ This time none of the intertwining operators has a pole, 
so the sum vanishes at $s= \frac 12.$ 

Alternatively, having established that $M(w_5,  [-1, 3s-1 + \chi_0]) = -1+O(s-\frac 12),$ 
We can deduce vanishing of the Eisenstein series from the fact that it is holomorphic and satisfies the 
functional equation
$$
E_P([-1, 3s-1+\chi_0]) 
= E_P( [-1, 2-3s+\chi_0])\circ 
M(w_5, [-1, 3s-1+\chi_0]).
$$
\end{proof}

\section{Siegel-Weil type identities}
In this section we prove identities relating degenerate Eisenstein series 
induced from the two different parabolic subgroups 
of $G_2.$ Such identities are sometimes called Siegel-Weil $n$th term identities. 
A conceptual explanation for their existence comes from embeddings
of degenerate induced representations into principal series representations induced 
from the Borel, together with the symmetry of the principal series.  We first prove a technical 
result which extends this philosophy to flat sections and Eisenstein series.
\subsection{Surjectivity property of intertwining operators}
In the next few sections we consider an alternate normalization of the 
intertwining operator, which is different than the one considered in section
\ref{s: ind reps, int ops, poles}.

\label{ss: stand int ops and alt normalization}
Take $\chi \in \X_{M,v} \subset \X_{T,v}.$ 
Then the standard intertwining operator $M(w_\alpha, \chi + \frac \alpha2)$ 
maps $I_B^G( \chi+ \frac \alpha 2)$ to $I_P^G (\chi) \subset I_B^G (\chi - \frac \alpha 2).$ 
We sketch a proof that the map is surjective. 
Take any $f \in I_P^G (\chi)$ and let $f_s^\circ $ denote the spherical vector in
$\Ind_B^G (s+\frac 12)\alpha.$
Then let $\wt f := f\cdot f_s^\circ ,$ which lies in the space 
$I_B^G (\chi + s\alpha).$ It follows immediately 
from the integral formula for the standard intertwining operator
that 
$$
M(w_\alpha,\chi + s\alpha)\wt f = f M(w_\alpha,s\alpha).f_s^\circ  
=f \frac{\zeta(2s)}{\zeta(2s+1)} f_{-s}^\circ .
$$
But $f_{-1/2}^\circ $ is just the constant function 1, so when $s=1/2$
we obtain a nonzero scalar multiple of $f.$ 

In the global setting, this fails because $\zeta(2s)$ has
a pole at $s=1/2.$ 
Let $$
M^*( w_\alpha, \chi ) 
= (\la \alpha^\vee, \chi \ra _{\un}-1) M( w_\alpha, \chi).
$$
(Recall that  $\la \alpha^\vee, \chi \ra _{\un}\in \unr{GL_1}$ 
which has been identified with $\C.$) 
Then $M^*(w_\alpha, \chi)$ 
 has no poles in $\Re(s) \ge 0,$
and for $\la \chi , \alpha^\vee \ra =0$
maps $I_B^G(\chi+ \frac \alpha 2)$ (global induced rep now) 
surjectively onto $I_P^G (\chi) \subset I_B^G ( \chi - \alpha/2).$ 

\begin{prop}\label{prop: flat sections B to flat sections P}
If $\Phi(T,M) = \{ \pm \alpha\}$
and $f$ is a flat section of $\c I_P^G\chi_0,$
then 
there exists   a flat section $\wt f$ of $\c I_B^G \chi_0$
such that 
 $M^*(w_\alpha, \chi+\frac \alpha 2).\wt f_{\chi + \frac\alpha 2} = f_\chi$
 for all $\chi \in \chi_0 + \X_{M,\un}.$  
\end{prop}
\begin{proof}
As before, we construct $\wt f$ by taking the product of $f$ and the 
spherical vector in $\Ind_B^G (s+\frac12)\alpha.$
(Note that $\X_{T,\un} = \X_{M,\un} + \C \alpha$.)
Then
for $\chi \in \X_M$ and $s \in \C$
$$M^*(w_\alpha, \chi+s\alpha).\wt f_{\chi + s\alpha}
= f_\chi \cdot (2s-1) \cdot \frac{\zeta(2s)}{\zeta(2s+1)}\cdot f_{-s}.
$$
Again, $f_{-1/2}$ is the constant function $1,$ so we get a nonzero scalar multiple of $f_{\chi}.$ 
\end{proof}
\begin{rmk} For each fixed $\chi,$ the operator
$M^*(w_\alpha, \chi+\frac\alpha2)$ maps $I_B^G(\chi + \frac \alpha 2)$ 
into $I_P^G(\chi ) \subset I_B^G(\chi - \frac \alpha2).$ The key point is that this extends to a map from 
flat sections of $\c I_B^G(\chi_0)$ to flat sections of $\c I_P^G(\chi_0).$
\end{rmk}

\begin{prop}
Take $P = MU$ 
with $\Phi(T,M) = \{ \pm \alpha\}$ 
and $f $ a flat section of $\c I_P^G(\chi_0).$
Assume $\la \alpha^\vee, \chi_0\ra = 0.$
Choose $\wt f $ a flat section of $\c I_B^G(\chi_0)$
such that $M^*(w_\alpha, \chi+\frac\alpha2).\wt f|_{\chi_0 + X_{M,un}} = f.$ 
Define $E_B^{\alpha*}\wt f(g,\chi) =( \la \chi , \alpha^\vee \ra_\un-1)
E_B.\wt f(g,\chi+ \frac \alpha2) $
Then 
$$
E_B^{\alpha*}.\wt f(g,\chi)= E_P.f(g, \chi), \qquad ( \forall \chi \in \chi_0+ X_{M,un}).
$$ 
\end{prop}
\begin{proof}
The two sides have the same constant term, namely
$$
\sum_{w \alpha > 0 } 
M(w, \chi-\frac \alpha 2).M^*(w_\alpha, \chi+\frac \alpha 2).\wt f = \sum_{w \alpha > 0 } M(w,\chi- \frac \alpha 2).f.
$$ Hence their difference is both a cusp form and a linear 
combination of Eisenstein series. As such, it is zero.
\end{proof}

\subsection{An identity of ramified Eisenstein series on $G_2$}
\label{sec: identity of ramified Eisenstein series}
\begin{thm}\label{thm: identity, quad case}
Take $\eta \in \X_{GL_1},$ nontrivial quadratic, and  $\wt f $ a flat section 
of  $\c I_B^G[0, \eta].$
Let $f$ be the flat section of $\c I_P^G[0,\eta]$
determined by $\wt f$ as in proposition  \ref{prop: flat sections B to flat sections P}.
The standard intertwining operator
$M(w_\alpha w_\beta)$ is an isomorphism at $[1, \eta-1].$
Let $\wt h$ be the flat section of  $\c I_B^G[\eta, 0]$ satisfying 
$\wt h_{[\eta, 1]} = M(w_\alpha w_\beta).\wt f_{[1, \eta-1]}.$
Let $h$ be the flat section  of $\c I_Q^G[\eta, 0]$
determined by $\wt h$ as in proposition \ref{prop: flat sections B to flat sections P}.
Then $E_P.f$ has a pole at $[0, \eta+\frac 12],$
$E_Q.h$ has a pole at $[\eta+\frac 12, 0],$
and the two residues are the same.
\end{thm}
\begin{proof}
Note that $\alpha = [2, -3]$ and $\beta = [-1,2].$
Note also that 
$$
w_\alpha w_\beta [0,1] = [-1,2] = \beta \equiv \varpi_1 \pmod 2,
\qquad
w_\beta w_\alpha \alpha =[0,1].
$$
It follows that 
$$
w_\alpha w_\beta [0, s+\eta] + u \frac \alpha 2 = [u+\eta, 0] + s \frac \beta 2.
$$
Now, 
$$
E_P[0, s+\eta].f = (u -\frac 12)E_B([0,s+\eta]+u \frac \alpha 2)\wt f|_{u=\frac 12}, 
\qquad 
E_Q[u+\eta, 0].h = (s-\frac12) E_B([u+\eta,0]+ s\frac \alpha 2)\wt h|_{s=\frac 12},
$$
and 
$$\begin{aligned}
E_B([0,s+\eta]+u \frac \alpha 2)\wt f &= E_B([u+\eta ,0 ]+s \frac \alpha 2)M(w_\alpha w_\beta, [0,s+\eta]+u \frac \alpha 2)\wt f\\
&= E_B([u+\eta,0]+ s\frac \alpha 2)(\wt h+ \text{ higher order terms.})
\end{aligned}
$$
It follows that the residue of $E_P.f$ at $s = \frac 12$ and that of $E_Q.h$ at $u=\frac 12$
are two different expressions for the value of the meromorphic continuation 
of $(s-\frac 12)(u-\frac 12)E_B([0, s+\eta]+ u\frac \alpha 2).\wt f$ to $s=u=\frac 12.$ 
\end{proof}

\subsection{An identity of unramified Eisenstein series on $G_2$}
\label{sec: identity of unram Eis ser}
In this section we prove an intriguing identity between unramified $G_2$ Eisenstein series. 
This identity is presumably related to the ``second term identity'' obtained by D. Jiang in unpublished work.

Define $[a,b]$  as before and write $f^\circ_{[a,b]}$ for the spherical vector in the corresponding induced representation. 
Let 
$c(s) = \frac{\zeta(s)}{\zeta(s+1)}.$ 
Identify $\X_{GL_1,\un}$ with $\C$ as usual. Write $M(w)$ for the standard intertwining operator attached to $w.$ 
Recall that
$$
M(w).f^\circ_{[a,b]} = \prod_{\alpha > 0, w \alpha < 0 } 
c( [a,b]\circ \alpha^\vee).f_{w\cdot [a, b]}^\circ
$$
Also, $c(s)$ has a simple zero at $s=-1$ a simple pole at $s=1$ and is holomorphic and nonvanishing at each other integer.
Write $c_{i,j}$ for the $i$th Laurent coefficient at $j,$ so that 
$$
c(j+s) = \sum_{i = \ord_j(c)}^\infty c_{i,j} s^i, \qquad \ord_j(c) = \begin{cases}
1, & j = -1, \\
-1, & j = 1, \\
0, & j \in \Z\ssm \{ \pm 1\}.
\end{cases}
$$
Also $c(s) c(-s) = 1$ and $c(0) = -1,$ which implies that 
$$
c_{1,-1} = -c_{-1,1}^{-1}\qquad 
c_{2,-1} = -\frac{c_{0,1}}{c_{-1,1}^2}\qquad
c_{3,-1} = \frac{c_{1,1}}{c_{-1,1}^2}-\frac{c_{0,1}^2}{c_{-1,1}^3}\qquad 
c_{0,0}=-1,\qquad c_{2,0} = -\frac 12 c_{1,0}^2 
$$
\begin{thm}\label{thm: id of unr eis ser}
The meromorphic function 
$$
E_Q f^\circ_{[3u+2,-1]} - \frac 13 c(u) c(3u-1) E_P f^\circ_{[-1, u+1]}.
$$
vanishes identically at $u=0.$
\end{thm}
\begin{rmk}
It would be interesting to make sense of this identity in terms of the functional equations
of the Eisenstein series $E_B.f^\circ,$ which are parametrized by the Weyl group 
of $G_2.$ However, this is not so trivial, 
as $[3u+2, -1]$ is not in the same Weyl orbit as $[-1, u+1]$ for $u \ne 0.$ 
\end{rmk}
\begin{proof}
It suffices to prove that the constant term vanishes.  We compute the constant terms of  $E_Q f^\circ_{[3u+2,-1]}$ and $E_P f^\circ_{[-1, u+1]}.$ 
We have 
$$\begin{aligned}
(E_Qf^\circ_{[3u+2,-1]})_B
=& f^\circ_{[3u+2,-1]} + c(3u+2) f^\circ_{[-2-3u, 3u+1]}+ \\
&+c(3u+1)c(3u+2)\left( f^\circ_{[6u+1,-1-3u]}+c(3u)c(3u-1)c(6u+1)f^\circ_{[1-3u,-1]}\right)\\
&+ c(3u+1)c(3u+2)c(6u+1) \left( f^\circ_{[-6u-1,3u]} + c(3u) f^\circ_{[3u-1,-3u]}\right)
\end{aligned}
$$
$$\begin{aligned}
(E_Pf^\circ_{[-1,u+1]})_B
&=\ff{-1,u+1} + c(u+1) \ff{3u+2,-1-u} \\
&+c(u+1)c(3u+2) \ff{-3u-2,2u+1} + c(u+1)c(3u+2)c(2u+1)\ff{3u+1,-2u-1} +\\
&+  c(u+1)c(3u+2)c(2u+1)c(3u+1)
\left(\ff{-3u-1,u}+c(u)\ff{-1,-u}
\right)
\end{aligned}$$
The terms are grouped 
 according to which $T$-eigenspace they reside in when $u=0.$ 
The proof is by direct computation. For each $T$-eigenspace
we consider the terms in the Laurent expansion up to $O(u).$
For example take the character $[-1,1]$ of $T.$ 
In $(E_Qf^\circ_{[3u+2,-1]})_B$ it does not appear. 
In $- \frac 13 c(u) c(3u-1)(E_Pf^\circ_{[-1,u+1]})_B$ the contribution is 
$- \frac 13 c(u) c(3u-1) f^\circ_{[3u+2,-1]},$ which 
nontrivial, but vanishes at $u=0$ because $c(3u-1)$ vanishes at 
$u=0$ and none of the other terms has a pole at $u=0.$ 
This is a fairly simple example.

We consider one additional example, which is a little more complex. 
 Recall that 
$$
f_{[a,b]}^\circ(ntk) = |t^{\varpi_1}|^{a+1} |t^{\varpi_2}|^{b+1}, 
\qquad (n \in N(\A), \ t \in T(\A), \ k \in K).
$$
Hence 
$$
f_{[a,b]+u[c,d]}^\circ (ntk)
= f_{[a,b]}(ntk) \cdot
\left( 
\sum_{m=0}^\infty 
\frac{1}{m!} [(\log |t^{\varpi_1}|c  + \log|t^{\varpi_2}|d)\cdot u]^m
\right).
$$
Write $\la c, d\ra$ for the mapping 
$$
(ntk) \mapsto  \log |t^{\varpi_1}|c  + \log|t^{\varpi_2}|d, \qquad (c,d \in \Z, n \in N(\A), \ t \in T(\A), k \in K).
$$
Then we have 
$$
f_{[a,b]+u[c,d]}^\circ = f_{[a,b]}^\circ 
\cdot \sum_{m=0}^\infty \frac{\la c, d \ra^m}{m!} u^m.
$$
Note that $\la c_1, d_1 \ra + \la c_2, d_2 \ra= \la c_1+c_2, d_1+d_2\ra.$
We compare the contributions to the constant terms of $E_Q f^\circ_{[3u+2,-1]}$ and $E_P f^\circ_{[-1, u+1]}$ 
which lie in the $[1,-1]$ eigenspace when $u=0.$ The relevant portion of $\frac 13 c(u)c(3u-1)(E_P f^\circ_{[-1, u+1]})_B$ is 
\begin{equation}
\tag{[1,-1],P}\label{[1,-1],P}
\begin{aligned}
\frac 13 
&c(u) c(3u-1)c(u+1)c(3u+2)c(2u+1)\ff{3u+1,-2u-1} ,
\end{aligned}
\end{equation}
while the relevant portion of 
$E_Q f^\circ_{[s+2,-1]}$ is 
\begin{equation}\label{[1,-1]Q}
\tag{[1,-1],Q}
c(s+1)c(s+2)\left( f^\circ_{[2s+1,-1-s]}+c(s)c(s-1)c(2s+1)f^\circ_{[1-s,-1]}\right).
\end{equation}
It makes sense to simplify \eqref{[1,-1]Q} before substituting $s=3u.$ However, 
notice that once this substitution is made, the factor of $c(3u+2)$ appears in 
both 
\eqref{[1,-1],P} and \ref{[1,-1]Q}. So, we may omit it from both. 
We expand the remainder of \eqref{[1,-1],P}.
\begin{equation}\label{[1,-1,P,alt]}
\begin{aligned}
\frac 13 
&c(u) c(3u-1)c(u+1)c(2u+1)\ff{3u+1,-2u-1} 
\\ &=
\frac13(-1+c_{1,0} u +O(u^2))(-\frac 3{c_{-1,1}} u -9 \frac{c_{0,1}}{c_{-1,1}^2} u^2 +O(u^3))(\frac{c_{-1,1}}{u} + c_{0,1} +O(u))
\times \\
&\times
 (\frac{c_{-1,1}}{2u}+c_{0,1} + O(u))
 f_{[1,-1]}^\circ 
 (1+ \la3,-2\ra  u + O(u^2))\\
 &=f_{[1,-1]}^\circ \left( 
 \frac{c_{-1,1}}{2u}  -\frac{c_{-1,1}c_{1,0}}2 + 3c_{0,1} +\frac{c_{-1,1}}2 \la 3 , -2 \ra + O(u)
 \right)
\end{aligned}
\end{equation}
We simplify the expression in brackets in  \eqref{[1,-1]Q}
$$
\begin{aligned}
 &f^\circ_{[2s+1,-1-s]}+c(s)c(s-1)c(2s+1)f^\circ_{[1-s,-1]}\\
 &\qquad = f_{[1,-1]}^\circ \left(
 1+ s\la 2, -1\ra + (-1 + c_{1,0} s)(-\frac{s}{c_{-1,1}}- \frac{c_{0,1}s^2}{c_{-1,1}^2})(\frac{c_{-1,1}}{2s}+c_{0,1})(1+\la -1, 0\ra s)+O(s^2)
 \right)\\
 & \qquad = f_{[1,-1]}^\circ 
 \left( 1+  s\la 2, -1\ra + \frac 12 ( 1+ s [ -c_{1,0} + \frac{3c_{0,1}}{c_{-1,1}} + \la -1,0\ra ])
 \right)\\
 & \qquad = f_{[1,-1]}^\circ 
 \left( \frac 32+  \frac s2( \la 3, -2\ra   -c_{1,0} + \frac{3c_{0,1}}{c_{-1,1}}  )
 \right)
\end{aligned}
$$
Multiplying by $c(s+1) = (c_{-1,1}/s+ c_{0,1} + O(s))$ yields 
$$
\begin{aligned}
& f_{[1,-1]}^\circ 
 \left( \frac {3c_{-1,1}}{2s}+  \frac {c_{-1,1}}2\left(\la 3, -2\ra   -c_{1,0} + \frac{3c_{0,1}}{c_{-1,1}} \right)+\frac 32 c_{0,1} +O(s)
 \right)\\
 &= f_{[1,-1]}^\circ 
 \left( \frac {3c_{-1,1}}{2s}+  \frac {c_{-1,1}}2\la 3, -2\ra   - \frac {c_{1,0}c_{-1,1}}2 + {3c_{0,1}}+O(s)
 \right),
 \end{aligned}
 $$
 and now substituting $s=3u$ yields \eqref{[1,-1,P,alt]}.
 The other eigenspaces are treated similarly.
\end{proof}

\section{Global zeta integral involving a $Q$-Eisenstein series}
In this section we consider the degenerate Eisenstein series $E_Q.h$ on $G_2(\A),$ where $Q= LV$ is the parabolic subgroup whose Levi contains
the  root subgroup attached to the long simple root. We let $H_\rho$ be a quasisplit subgroup of $G_2$ of type $A_2$ defined as in \cite{su21adjoint}.
Thus $\rho \in F^\times$ and $H_\rho$ is isomorphic to $SL_3$ if $\rho$ is a square, and a quasisplit unitary group attached to the corresponding 
quadratic extension if $F$ is $\rho$ is a nonsquare. For any normalized character
$\chi_0$ of $L(F) \bs L(\A)$ and any flat section $h$ of $I=\c I_Q^{G_2}(\chi_0)$ the restriction of $E_P.h$ to $H_\rho(\A)$ 
is a smooth function of moderate growth on $H_\rho(F) \bs H_\rho(\A),$ so it may be integrated against a cuspform 
on $H_\rho.$ In the split case, this integral is identically zero for all $\chi_0, h,$ as shown in proposition 4.6 of \cite{Ginzburg-Jiang}. 
 We extend the same idea to the general case.  
 
 First we need to analyze 
$Q(F) \bs G_2(F)/H_\rho(F).$
Following \cite{su21adjoint}, we regard $G_2$ as a subgroup of (split) $SO_8$ preserving the 
 quadratic form $\b x \mapsto {}_t\b x\cdot \b x.$ Here
$_t$ is the ``other transpose'' (as in \cite{su21adjoint}).  Then $G_2$  contained in the 
group $SO_7$ preserving the subspace $V_0:= \{ \b x ={} ^t \bbm x_1&\dots& x_8\ebm : x_4 = x_5 \}.$ 
The group $H_\rho$ is the stablizer of $v_\rho: = {}^t\bbm 0&0&1&0&0&\rho& 0&0\ebm $ in $G_2.$ 
Thus $G_2(F)/H_\rho(F)$ may be identified with the orbit $O_\rho:= G_2(F) \cdot v_\rho,$
and $Q(F) \bs G_2(F)/H_\rho(F)$ can be identified with the set of $Q(F)$-orbits in $O_\rho.$

\begin{lem}
The orbit 
$
O_\rho$ is equal to 
$ \{ \b x \in V_0:  \ _t\b x \cdot \b x = 2\rho\}.$ If $\rho$ is not square it is a union of two $Q(F)$ orbits 
namely 
$$
 \{ \b x = \bbm x_1& \dots & x_8 \ebm \in O_\rho:  x_8 \ne 0\},\qquad 
  \{ \b x = \bbm x_1& \dots & x_8 \ebm \in O_\rho:  x_8 = 0\}.
$$
If $\rho = a^2$ then $ \{ \b x = \bbm x_1& \dots & x_8 \ebm \in O_\rho:  x_8 \ne 0\}$ is still a single $Q(F)$-orbit, 
while $$\{ \b x = \bbm x_1& \dots & x_8 \ebm \in O_\rho:  x_8 = 0\}$$ is a union of three orbits, viz. 
$$\{ \b x = \bbm x_1& \dots & x_8 \ebm \in O_\rho:  x_8 = 0, \{ x_6, x_7\} \ne \{ 0\}\},$$
$$\{ \bbm x_1&x_2&x_3 & a&a& 0&0&0\ebm\} , \qquad \text{ and } 
\{ \bbm x_1&x_2&x_3 & -a&-a& 0&0&0\ebm\}.
$$
\end{lem}
\begin{proof}
It's clear that each element $\b x $ of $G_2(F) \cdot v_\rho$ 
satisfies $_t \b x \cdot \b x = 2\rho.$ 
Let $v_\rho' = {}^t\bbm \rho& 0& \dots & 0 & 1 \ebm.$ 
Then a suitable representative for $s_\alpha s_\beta s_\alpha$ 
maps $v_\rho$ to $v_\rho'.$ One checks that the second, third, fourth, sixth, and seventh elements 
of the last column of an element of the standard maximal unipotent subgroup of $G_2(F)$ can be chosen 
arbitrarily and it follows that the $Q(F)$ orbit of $v_\rho'$ (and also of $v_\rho$) contains every element of 
$\{ \b x = \bbm x_1& \dots & x_8 \ebm \in O_\rho:  x_8 \ne 0\}.$ 
On the other hand, this subset is clearly $Q(F)$-stable, so it is the $Q(F)$ orbit of $v_\rho'.$ 
Likewise, a suitable representative of $s_\beta s_\alpha$ maps $v_\rho$ to 
$v_\rho''=^t\bbm 0&\rho&0& \dots& 0 & 1 &0 \ebm$ and the orbit of $v_\rho''$ under the Borel of $G_2(F)$ is 
readily seen to be all $\b x$ with $_t \b x \cdot \b x = 2\rho, \ x_8 = 0, \ x_7 \ne 0.$ Clearly, each vector with $x_6, x_7$
not both zero is $Q(F)$-equivalent to one with $x_7 \ne 0.$ It follows that 
$\{ \b x : \ _t \b x \cdot \b x = 2\rho, \ x_8 = 0, \{ x_6, x_7\} \ne 0 \}$ is the $Q(F)$ orbit of $v_\rho''.$ 
Recall that $\b x \in V_0$ forces $x_4 = x_5.$ If $x_6=x_7=x_8= 0$ we get 
$_t \b x \cdot \b x = 2x_4^2,$ so that $\b x : \ _t \b x \cdot \b x = 2\rho \implies \rho = x_4^2.$ If $\rho$ is 
not a square this is impossible. It follows that the two $Q(F)$ orbits already described exhaust all vectors 
with $\b x : \ _t \b x \cdot \b x = 2\rho.$ If $\rho = a^2$ we have the two additional subsets described. Each is 
readily seen to be a $Q(F)$-orbit. Let $X_{-\alpha}$ be the matrix with a one at positions $(2,1), (4,3)$ and $(5,3),$
a $-1$ at positions $(8,7), (6,5)$ and $(6,4),$ and zeros everywhere else. It spans the root subspace for the root $-\alpha.$
Let $x_{-\alpha}(r) = \exp( r X_{-\alpha}).$ (This is a polynomial formula, because $X_{-\alpha}$ is nilpotent.) 
Then 
$$
x_{-\alpha}(\pm a) \cdot v_{a^2} = \ ^t \bbm 0&0&1&\pm a& \pm a & 0&0&0\ebm,
$$
which proves that our two additional $Q(F)$-orbits are still in the $G_2(F)$-orbit of $v_{a^2}.$ Clearly, the four taken together 
exhaust $ \{ \b x \in V_0:  \ _t\b x \cdot \b x = 2\rho\}.$ This completes the proof. 
\end{proof}
\begin{cor}
If $\rho$ is not a square, then $\{e, s_\alpha s_\beta\}$ is a set of representatives for $Q(F) \bs G_2(F) / H_\rho(F).$
If $\rho = a^2,$ then $\{e, s_\alpha s_\beta, x_{-10}(a), x_{-10}(-a)\}$ is a set of representatives.
\end{cor}
\begin{proof}
Let each element act on $v_\rho$ and consider the explicit description of the $Q(F)$-orbits in $O_\rho.$
\end{proof}

\subsection{A certain period}
Let $\varphi_{32}: SL_2 \to G_2$ be the mapping 
$$
\bpm a& b \\ c & d \epm \mapsto 
\bpm
a&&&&&&b&\\
&a&&&&&&b\\
&&1&&&&&\\
&&&1&&&&\\
&&&&1&&&\\
&&&&&1&&\\
c&&&&&&d&\\
&c&&&&&&d\\
\epm.
$$
The image is contained in $H_\rho$ for every $\rho.$  (Cf. (2) on p. 195 of \cite{su21adjoint}.)  If we embed $H_\rho(F)$ into $GL_3(E)$ using the identification 
 $E^3$ with the six dimensional subspace of $F^8$ stabilized by $H_\rho$ as in \cite{su21adjoint} remarks 2, then $\varphi_{32}$ corresponds to the mapping 
 $$\bpm a&b\\c&d \epm \mapsto 
 \bpm a&& \frac{b}{2\tau}\\ &1&\\
 2\tau c && d 
\epm,$$
where $\tau \in E$ satisfies $\tau^2 = \rho.$  
From this point of view, the image of $\varphi_{32}$ is a smaller special unitary group, and we denote in $H'_\rho.$ 
We consider the period 
\begin{equation}\label{eq: period}
\c P(\varphi) = \iq{H'_\rho} \varphi(h) \, dh.
\end{equation}
We say that an automorphic representation $\pi$ is $H'_\rho$-distinguished if $\c P$ does not vanish identically on it.

\begin{prop}
\label{prop: E_Q orth to G cuspforms}
Let $h$ be a flat section of $\c I_Q^{G_2}(\chi_0)$ and $\varphi$ a cuspform defined on $H_\rho(\A).$
Assume that $s \in \X_{Q, \un}$ is not a pole 
of $E_Q.$ Then for a suitable measure on the space 
$H'_\rho(\A)\bs H_\rho(\A)$ we have
\begin{equation}\label{EQ Eis unfolded}
\iq{H_\rho} \varphi(g) E_Q.h(g, s) \, dg = \int\limits_{H'_\rho(\A)\bs H_\rho(\A)} h(s_\alpha s_\beta g, s) \c P( R(g).\varphi) \, dg.
\end{equation}
(Here, $R$ denotes right translation.) In particular, the integral vanishes identically on any cuspidal automorphic representation which 
is not $H'_\rho$-distinguished. If $H_\rho$ is split, then no cuspidal automorphic representation is $H'_\rho$-distinguished, 
and the integral \eqref{EQ Eis unfolded} is always zero.
\end{prop}

\begin{proof}
For $s$ in the domain of convergence for $E_Q,$ our integral is equal to 
$$
\sum_{ \gm \in Q(F) \bs G_2(F) / H_\rho(F) }\int_{H_\rho(F) \cap \gm^{-1} Q(F) \gm \bs H_\rho(\A)} h(\gm g,s) \varphi(g) \, dg.
$$ 
If $\gm$ is the identity, or if $\rho = a^2$ and $\gm = x_{-\alpha}(\pm a)$ then $H_\rho \cap \gm^{-1} Q \gm$ contains the unipotent radical of 
the standard  Borel subgroup of $H_\rho.$ Hence the function $g \mapsto h(\gm g,s)$ is invariant by this subgroup on the left, 
and the integral over $H_\rho(F) \cap \gm^{-1} Q(F) \gm \bs H_\rho(\A)$ factors through the mapping that sends $\varphi$ to its constant
term-- which is zero. If $\gm = s_\alpha s_\beta,$ then 
$H_\rho \cap \gm^{-1} Q \gm =H'_\rho,$
and $\gm H_\rho' \gm^{-1}$ is contained in the derived group of the Levi $L$ of $Q.$ 
Hence the function $g \mapsto f(s_\alpha s_\beta g, s)$ is invariant by $H_\rho'(\A)$ on the right.
  We are reduced to showing that 
$$
\int\limits_{H'_\rho(F) \bs H_\rho(\A)}  h(s_\alpha s_\beta g, s) \varphi(g) \, dg 
= \int\limits_{H'_\rho(\A) \bs H_\rho(\A)} h(s_\alpha s_\beta g_1, s) \iq{H'_\rho}\varphi(h'_\rho g_1)\, dh'_\rho \, dg_1 
$$
for a suitable measure $dg_1.$ Here, $dg$ and $dh'_\rho$ are the Haar measures on the respective groups. 
  The existence of a suitable 
invariant measure on each of the corresponding {\it local} homogeneous spaces $H'_\rho(F_v)\bs H_\rho(F_v)$ follows from proposition 
4.3.5 of \cite{Bump-GreyBook}. The existence of $dg_1$ then follows easily from expressing $dg$ and $dh'_\rho$ in terms of 
the Haar measures on the local groups.

If $H_\rho$ is split, then $H'_\rho$ is a nonstandard Levi subgroup and is conjugate to a standard Levi subgroup. It suffices to show that the 
period along the standard Levi subgroup vanishes. As noted in \cite{Ginzburg-Jiang} this follows easily from cuspidality.
\end{proof}
\begin{rmk}
We may regard $H'_\rho$ as $SU_{1,1}  \subset U_{1,1}  \subset U_{2,1}$ and one could 
enlarge our period $\c P,$ to an integral over $\quo{U_{1,1}}$ against a character. Such periods are considered in 
 \cite{GeRoSo}, where they are used to characterize the image of the theta lifting.  
 Recall that the $\Ad'$ $L$ function of a representation in the image of this lifting should have a pole at $s=1.$ 
 So, the emergence of the $H'_\rho$ period as an obstruction to proving holomorphy in general 
makes perfect sense.
\end{rmk}

\subsection{Conclusions regarding poles of global zeta integrals}
\begin{thm}
Suppose that $\varphi$ generates a cuspidal automorphic representation which is not $H'_\rho$ distinguished. 
Then the global zeta integral $I(s, \varphi, f)$ defined in \eqref{eq: def of I(s, vph, f)} 
has no poles in the half plane $\Re(s) \ge \frac 12,$ except for 
a simple pole
at $s=\frac 23$ which can occur only 
when $\chi_0$ is cubic. 
Moreover, if $\chi_0$ is quadratic, then $I(\frac 12, \varphi, f )=0$ for all $\varphi$ and $f.$ 
\end{thm}
\begin{rmk}
For these purposes the trivial character is both cubic and quadratic.
\end{rmk}
\begin{proof}In the case when $\chi$ is unramified this is a slight 
refinement of proposition 4.6 of \cite{Ginzburg-Jiang}.
We know from theorems \ref{thm: poles of EP, chi0 trivial} and \ref{thm: poles in cubic and quadratic case}
that the only possible poles in $\Re(s)\ge \frac 12$ occur at $s=\frac 23$ when $\chi_0$ is trivial, quadratic, or 
cubic, 
and at $s= 1$ when $\chi_0$ is trivial. 
When $H_\rho$ is split and $\chi_0$ is trivial, the possibility of a pole at $s=1$ or a double pole 
at $s=2/3$ are ruled out by \cite{Ginzburg-Jiang}.
When $\chi_0$ is quadratic, the possibility of a pole at $\frac 23$ is ruled out 
by the same argument, using theorem 
\ref{thm: identity, quad case} and proposition \ref{prop: E_Q orth to G cuspforms}.
The vanishing at $\frac 12$ when $\chi_0$ is quadratic 
follows from proposition \ref{prop: vanishing at 1/2 when quad}.
\end{proof}

\section{Nonvanishing of local zeta integrals}
In this section we prove that for any fixed $s_0 \in \C$ there is a choice of data such that the 
local zeta integral for the adjoint $L$ function is nonzero at $s_0.$ It then follows that any 
pole of the partial adjoint $L$ function would give rise to a pole of the global zeta integral, 
except possibly at poles of the ``normalizing factor''
$L^S(3s, \chi)L^S(6s-2, \chi^2) L^S(9s-3, \chi).$
We take up some notations from \cite{su21adjoint}: 
 for fixed $\rho \in F,$ $H_\rho$ is a subgroup 
 of $G_2$ isomorphic to $SL_3$ 
 over any field in which $\rho$ is a square.
 We equip it with a choice of Borel subgroup 
 $B=TN$ 
 where $T$ is a maximal torus and $N$ is 
 a maximal unipotent. (This departs from our previous usage of $B$ and $T$ for the Borel and torus of $G_2.$)
  Also $w_2$ is the second simple reflection 
 in $G_2$ (attached to the long simple root), 
 and $N_2$ is a two dimensional unipotent subgroup, 
 with the property that 
 $H_\rho \cap w_2^{-1} P w_2 = N_2 T_\s,$
 where $T_\s$ denotes the one dimensional 
 maximal $F$-split torus contained in the standard 
 Borel of $H_\rho.$ 
 Finally, $\psi_N$ is a certain generic character
 of $N.$ 
 Details are found in \cite{su21adjoint}.
It is convenient to identify $\X_{M,\un}$ with $\C$ via the map 
$s \mapsto \delta_P^s.$

For $\rho \in F,$ $\pi$ a irreducible admissible $\psi_N$-generic representation of $H_\rho,$  
$W$ in the $\psi_N$-Whittaker model $\c W_{\psi_N}(\pi)$ 
of $\pi,$ 
and 
$f \in \Flat(\chi_0)$ we define 
$$
I(W, f; s)= \int_{N_2\bs H_\rho} W(g) f_s( w_2 g) \, dg.
$$
\begin{thm}\label{thm: nonvanishing of local zetas}
For any $\rho \in F^\times,$ any irreducible admissible generic representation $\pi$ of $H_\rho,$
any  $\chi_0 \in \X_{M,0},$ 
and any fixed $s_0 \in \C,$
there exist $W \in \c W_{\psi_N}(\pi)$
and $f \in \Flat(\chi_0)$
such that 
$I(W, f; s_0) \ne 0.$ 
\end{thm}
\begin{proof}
Expressing the Haar measure on $H_\rho$ 
as a suitable product measure on the open Bruhat cell 
yields
$$
I(W, f; s) = \int_{N_2 \bs N} \int_{T} \int_{N^-} W( ntn^-)
f_s( wntn^{-}) \delta_B^{-1}(t) \, dn^-\, dt\, dn.
$$
We may  identify $N_2\bs N$ with 
the complementary subgroup
$$
N_2' := \left\{ \begin{pmatrix} 1&r&-\frac{r^2}2\\ &1&-r\\&&1\end{pmatrix}: r \in F \right \},
$$
which is $T_\s$-stable,
and fix a fundamental domain $[T_\s \bs T]$ 
for $T_\s$ in $T,$ to express $I(W, f; s)$ 
as 
$$
\int_{N_2'} \int_{T_\s}\int_{[T_\s\bs T]} \int_{N^-} W( nt_\s t'n^-)
f_s( w_2nt_\s t'n^{-}) \delta_B^{-1}(t_\s t') \, dn^-\, dt'\, dt_\s\,dn$$
$$=  \int_{T_\s}t_\s^{w_2 \chi_0 + w_2 s -3 \gm}
\II(W, f; t_\s, s)\, d t_\s
,$$
where $\gm$ is the common restriction 
of the two simple roots of $H_\rho$ to $T_\s$
(so that $\delta_B^{-1} (t_\s) = t_\s^{-4\gm}$)
and 
$$\II(W, f; t_\s,s):=
\int_{N_2'}\int_{[T_\s\bs T]} \int_{N^-} W( t_\s nt'n^-)
f_s( w_2nt'n^{-}) \delta_B^{-1}(t') \, dn^-\, dt'\, dn.
$$
Note that 
the function $(n, t', n^- ) \mapsto P w_2 n t' n^-$ 
is a continuous injection of $N_2' \times [T_\s \bs T]
\times N^-$ into $P \bs G_2.$ 
Now fix 
 $s_0,$ and let $\phi_f(n, t', n^- ):= \delta_B(t')^{-1}f_{s_0}( w_2nt'n^{-}),$
 which we view 
as a ``test function'' on $N_2' \times [T_\s \bs T]
\times N^-.$  First assume  that $F$ is nonarchimedean.
Then for any smooth function $\phi_1$ of compact
support defined on $N_2' \times [T_\s \bs T]
\times N^-$ we can choose $f$ so that $\phi_f = \phi_1.$ 
But then 
$$
I(W, f; s_0) = 
I'(\phi_f*W; s_0),
$$
where $*$ denotes the action by convolution and 
$I'(W, s)$
is defined for $W\in \c W_{\psi_N}(\pi)$ and $ s\in \C$
by $$
I'(W, s) := \int_{T_\s} W(t_\s) t_\s^{w_2 \chi_0 + w_2 s -3 \gm}
\, dt_\s.
$$
Hence $I(W, f; s_0)$ vanishes 
for all $f \in \Flat(\chi_0), W \in W_{\psi_N}(\pi)$ if and only if $I'(W; s_0)$ vanishes for all $W \in W_{\psi_N}(\pi).$

But now let 
$\phi_2$ be a Schwartz
function on $F$ and $x: F \to H_\rho$ an embedding into $N$ 
chosen so that 
$t_\s x(a) t_\s^{-1} = x(t_\s^\gm a)$ and 
$\psi_N(x(a)) = \psi(a).$ 
Then 
$$\begin{aligned}
I'((\phi_2\circ x)*W, s)&= \int_{T_\s} \int_F W(t_\s x(a))  t_\s^{w_2 \chi_0 + w_2 s -3 \gm}\phi_2(a) \, da \, dt_\s\\&= \int_{T_\s} \int_F W(t_\s)  t_\s^{w_2 \chi_0 + w_2 s -3 \gm}\psi(t_\s^\gm a)\phi_2(a) \, da \, dt_\s\\
& = \int_{T_\s}  W(t_\s)  t_\s^{w_2 \chi_0 + w_2 s -3 \gm}\widehat \phi_2(t_\s^\gm) \, dt_\s
\end{aligned}
$$
Notice that $\widehat \phi_2$ is a Schwartz function which 
can be chosen arbitrarily. Clearly, we can now choose
$W \in W_{\psi_N}(\pi)$ which does not 
vanish identically on $T_\s$ and then choose
$\phi_2$ so that $I'((\phi_2\circ x)*W, s_0) \ne 0.$ 

This completes the proof in the nonarchimedean case. 
In the archimedean case, the same argument shows
 that the 
mapping 
$(W, f) \mapsto I(W, f; s_0)$ does not vanish identically 
on $\c W_{\psi_N}(\pi) \times \sInd_P^G (\chi_0 + s_0),$
where $\sInd$ denotes smooth induction, as opposed to $K$-finite induction. But since the space of $K$-finite vectors 
is dense in the smooth induced representation, it then 
follows that
$(W, f) \mapsto I(W, f; s_0)$ can not vanish identically on 
 $\c W_{\psi_N}(\pi) \times\Ind_P^G (\chi_0 + s_0)$ 
 either, completing the proof in this case.
 \end{proof}

\section{Application to poles of the adjoint $L$ function}
\begin{thm}\label{thm: main}
Assume that $G$ is split, 
and let $S$ be a finite set of places, including all archimedean places such that 
$\pi_v$ and $\chi_v$ are unramified for $v \notin S.$ Then 
the partial twisted adjoint $L$ function 
$L^S(s, \pi, \Ad \otimes \chi)$ 
has no poles in the half-plane $\Re(s) \ge \frac 12,$ except possibly for a simple pole at 
$\Re(s) =1$ when $\chi$ is nontrivial cubic. 
If this pole is present, then it is inherited by the complete $L$ function  $L(s, \pi, \Ad \times \chi).$
Every other pole of  $L(s, \pi, \Ad \times \chi)$  in $\Re(s)\ge \frac 12$ is a zero of the Hecke $L$ function $L(s, \chi)$, and 
 a pole of $\prod_{v \in S} L_v(s, \pi_v, \Ad\times \chi_v).$  

\end{thm}
\begin{proof}
According to the results which we have 
proved so far, the global 
zeta integral $I(s, \varphi, f)$ has no poles in $\Re(s)\ge \frac 12$
except possibly for a simple pole at $\Re(s) = \frac 23$
which can occur only when $\chi$ is nontrivial cubic. 
By theorem \ref{thm: nonvanishing of local zetas}, 
these properties are inherited by the ratio of partial $L$ functions
$$
\frac{L^S(3s-1, \pi, \Ad' \otimes \chi)}{L^S(3s, \chi) L^S(6s-2, \chi^2) L^S(9s-3, \chi^3)},
$$
and then, since local $L$ functions are nonvanishing meromorphic functions, 
by  
$$
\frac{L^S(3s-1, \pi, \Ad' \otimes \chi)}{L(3s, \chi) L(6s-2, \chi^2) L(9s-3, \chi^3)}.
$$
The product $L(3s, \chi) L(6s-2, \chi^2) L(9s-3, \chi^3)$ has no poles 
in $\Re(s)\ge \frac 12$ except for the simple pole of $L(6s-2, \chi^2)$ 
at $s = \frac 12$ which occurs only if $\chi^2$ is trivial. But 
we have seen that $I(\frac 12 , \varphi, f)=0$ when $\chi$ is 
quadratic. 

This completes the proof of our assertions regarding the partial $L$ function. 
Since local $L$ functions are meromorphic but nonvanishing,  
passing from the partial to the completed $L$ function may introduce additional poles, but
 will not cancel the pole at $1$ 
in the case when it occurs.
On the other hand, it follows immediately from the definitions that $L(s, \pi, \Ad \times \chi) = L(s, (\pi \otimes \chi) \times \wt \pi)/L(s, \chi).$ 
By a result of Moeglin and Waldspurger \cite[Corollaire, p. 667]{MW}, the numerator has at most two simple poles, which occur at $0$ and $1$ and occur if and only if $\pi \cong \pi \otimes \chi$. (Note 
that this forces $\chi$ to be cubic.) Thus, any other poles which appear must be zeros of the denominator.
\end{proof}

\begin{rmk}
The expression  $L(s, \pi, \Ad) = L(s, \pi  \times \wt \pi)/\zeta(s)$
also gives us a shorter proof that $L(s, \pi, \Ad)$ is holomorphic and nonvanishing at $s=1,$ 
without appealing to theorem \ref{thm: id of unr eis ser} and proposition \ref{prop: E_Q orth to G cuspforms}
The functional equations of the global Rankin-Selberg $L$-function and the global 
Dedekind zeta function 
give a functional equation of $L(s, \pi, \Ad).$
By the functional equation, 
the set of poles of $L(s, \pi, \Ad)$ is symmetric as $s \mapsto 1-s.$ 
\end{rmk}
In the nonsplit case, the same argument gives the following result.
\begin{thm}\label{thm: main, nonsplit}
Assume that $G$ is nonsplit, 
and let $S$ be a finite set of places, including all archimedean places such that 
$\pi_v$ and $\chi_v$ are unramified for $v \notin S.$ Then 
the partial twisted  $L$ function 
$L^S(s, \pi, \Ad' \otimes \chi)$ 
has no poles in the half-plane $\Re(s) \ge \frac 12,$ except possibly at 
$s =1.$ 
The pole at $s=1$ can occur only when $\chi$ is trivial, quadratic, or cubic. 
If a pole occurs when $\chi$ is trivial or quadratic, then $\pi$ is $H'_\rho$-distinguished.
The pole at $s=1$ is at most a double pole when $\chi$ 
is trivial, and at most a simple pole when $\chi$ is nontrivial quadratic or cubic.
\end{thm}

\begin{rmk}
In the nonsplit case we have $L(s, \pi, \Ad' \times \chi) = L(s, \on{sbc}(\pi), \on{Asai} \times \chi)/L(s, \chi),$ where 
$\on{sbc}$ denotes the stable base change lifting from $U_{2,1}$ to $\Res_{E/F}GL_3,$ constructed in \cite{KimKrishOdd},
and $\Asai$ is the Asai representation. 
In the important special case $\chi = \chi_{E/F}$ this becomes
$L(s, \pi , \Ad) = L(s, \sbc(\pi), \Asai \times \chi_{E/F})/L(s, \chi_{E/F}).$
It is proved in  \cite{KimKrishOdd},
that  $L(s, \on{sbc}(\pi), \on{Asai})$ will have a simple pole at $s=1$ if 
$\sbc(\pi)$ is cuspidal, but also that $\sbc(\pi)$ need not be cuspidal.
Arguing as on p. 22 of \cite{GRSBook}, we may deduce that $L(s, \sbc(\pi), \Asai \times \chi_{E/F})$ is holomorphic and nonvanishing 
at $s=1,$ and deduce that $L(s, \pi, \Ad)$ has the same property, still provided that $\sbc(\pi)$ is cuspidal.
 \end{rmk}

\end{document}